\documentclass[12pt,letterpaper]{amsart}

\usepackage{amssymb,latexsym,amsmath,amsthm,graphicx}
\numberwithin{equation}{section}
\newtheorem{thm}{Theorem}[section]

\newtheorem{conj}[thm]{Conjecture}

\newtheorem*{claim}{Claim}

\newtheorem*{remark}{Remark}
\newtheorem*{definition}{Definition}

\newcommand{\pd}{\partial}

\begin{document}

\title[Schwarz potential]{Laplacian Growth, Elliptic Growth, and singularities of the Schwarz potential}

\author[Erik Lundberg]{Erik Lundberg}


\begin{abstract}
The Schwarz function has played an elegant role in understanding and in generating new examples of exact solutions to the Laplacian growth (or ``Hele-Shaw``) problem in the plane. 
The guiding principle in this connection is the fact that ``non-physical'' singularities in the ``oil domain'' of the Schwarz function are stationary, 
and the ``physical'' singularities obey simple dynamics.  
We give an elementary proof that the same holds in any number of dimensions for the Schwarz potential, introduced by D. Khavinson and H. S. Shapiro \cite{KhSh} (1989).
A generalization is also given for the so-called ``elliptic growth'' problem by defining a generalized Schwarz potential.

New exact solutions are constructed, and we solve inverse problems of describing the driving singularities of a given flow.
We demonstrate, by example, how $\mathbb{C}^n$-techniques can be used to locate the singularity set of the Schwarz potential.
One of our methods is to prolong available local extension theorems by constructing ``globalizing families''.
We make three conjectures in potential theory relating to our investigation.
\end{abstract}

\maketitle

\section{Introduction}

A one-parameter family of decreasing domains, $\{ \Omega_t \}$, in $\mathbb{R}^n$ solves the Laplacian growth problem with sink at ${\bf x}_0 \in \Omega_t$ if the normal velocity, 
$v_n$ of the boundary $\Gamma_t := \partial \Omega_t$ is determined by a harmonic Green's function, $P({\bf x},t)$, of $\Omega_t$ as follows.

\begin{equation}\label{LG}
\left\{
\begin{array}{l}
v_n|_{\Gamma_t}=-\nabla P \\
\Delta P = 0 $, in $ \Omega_t \\
P|_{\Gamma_t} = 0 \\
P({\bf x } \rightarrow {\bf x_0},t) \sim - Q \cdot K({\bf x}-{\bf x_0}),
\end{array}\right.
\end{equation}

where $K$ is the fundamental solution of the Laplace equation and $Q>0$ determines the suction rate at ${\bf x}_0$.  
We can also consider $Q<0$ for the case of a source ${\bf x}_0$ where injection occurs, 
but this problem is stable (approaching a sphere in the limit) and is sometimes called the ``backward-time Laplacian growth''.

This is a \emph{nonlinear} moving boundary problem,
ubiquitous as an ideal model (or at least, first approximation) of many growth processes in nature and industry.
We stress that we are considering here the ill-posed zero surface-tension case, where the interface can encounter a cusp.
The zero surface-tension case has attracted wide and growing attention mainly for two reasons (to be brief):
(i) it has direct connections to many other areas such as classical potential theory, integrable systems, soliton theory, and random matrices
(ii) it admits a miraculous complete set of explicit exact solutions in the two-dimensional case.

If the domains $\Omega_t$ are bounded, with $Q>0$ problem (\ref{LG}) actually produces a \emph{shrinking} boundary.  
We get a growth process if $\Omega_t$ contains infinity so $P$ then solves an \emph{exterior} Dirichlet problem.  
In such a situation, it is common to place the sink at infinity by prescribing asymptotics for $\nabla P$ 
so that the flux across neighborhoods of infinity is proportional to $Q$.  
In the two-dimensional case this can be realized in the laboratory using a \emph{Hele-Shaw cell}.  
Two sheets of glass are placed close together with a viscous fluid (``oil'') filling the void between them.  
A small hole is drilled in the center of the top sheet and an inviscid fluid (``water'') is pumped in at a constant rate.  
Then problem (\ref{LG}) serves as an ideal model for the boundary of the growing bubble of water.  
The harmonic function $P({\bf x},t)$, in this case, corresponds to the \emph{pressure} in the oil domain.  
In other physical settings modeled by (\ref{LG}), $P({\bf x},t)$ can be a probability, a concentration, an electrostatic field, or a temperature.
Because of the huge amount of literature, we are limited to citing an incomplete list of papers.
For a list of over 500 references, see \cite{GilHow}.

We are particularly attracted to this problem by the lack of explicit examples in dimensions higher than two.
The existence, uniqueness, and regularity theory are well-developed in arbitrary dimensions,
and in the plane there is an abundance of explicit, exact solutions.
In dimensions higher than two, the only examples are a shrinking sphere (in the case when the ``oil domain'' $\Omega_t$ is bounded) 
or the exterior of a homothetically growing ellipsoid (in the case $\Omega_t$ is unbounded).
The obvious explanation for this deficiency of explicit examples is a lack of conformal maps in higher dimensions (Liouville's Theorem) since exact 
solutions are usually described in terms of a time-dependent conformal map of the domain to the disk.  
However, exact solutions can be understood using a different tool from complex analysis, the Schwarz function (see \cite{Davis} and Section 2 below).  
The following theorem relates to the work of S. Richardson \cite{Rich} and was first stated in terms of the Schwarz function by R. F. Millar \cite{Mil}.
Also, the discussion given by S. Howison \cite{How} seems to have played an important role in popularizing the use of the Schwarz function in studies of Laplacian growth.
\begin{thm}[Dynamics of Singularities: $\mathbb{R}^2$]\label{SF}
Suppose a one-parameter family of domains $\Omega_t$ has smoothly-changing analytic boundary with Schwarz function $S(z,t)$.
Then it is a Laplacian growth if and only if
\begin{equation}\label{DSP}
\frac{\partial}{\partial t}S(z,t) =-4\frac{\partial}{\partial z}P(z,t)
\end{equation}
\end{thm}
The Schwarz function is only guaranteed to exist in a vicinity of a given analytic curve, 
and \emph{a priori} the domain of analyticity for its time-derivative is not any better.
Thus, it is surprising that for a Laplacian growth, $\frac{\partial}{\partial t} S(z,t)$
coincides with a function analytic throughout $\Omega_t$ except at the singularity prescribed at the ``sink''.
In other words, we can extract from equation (\ref{DSP}) the following elegant description of solutions to problem (\ref{LG}):
\emph{Singularities in $\Omega_t$ of the Schwarz function of $\partial \Omega_t$ do not move 
except for one simple pole stationed at the sink ${\bf x}_0$ which decreases in strength at the rate $-Q$.}
Since equation \ref{DSP} is given in physical coordinates rather than introducing a uniformized ``mathematical plane'', 
S. Howison \cite{How} has called it an \emph{intrinsic} description.
In recent papers, it is typical to see a combination of the Schwarz function and the conformal map used to derive solutions (e.g. \cite{AMZ}).
We will review some familiar examples in Section 4 and understand them completely in terms of Theorem \ref{SF}.

The Schwarz function has been partially generalized by D. Khavinson and H. S. Shapiro to higher dimensions by defining a ``Schwarz potential'', 
a solution of a certain Cauchy problem for the Laplace equation \cite{Shapiro}.  
In Section 2, we will review the definition of the Schwarz function and the Schwarz potential before proving  
the $n$-dimensional version of Theorem \ref{SF}.  
We also give a further generalization to the elliptic growth problem.
The rest of the paper is guided by Theorem \ref{LGSP}, which identifies, as
the main obstacle, the problem of describing (globally) the singularities of the Schwarz potential.  
In Section 4 we follow the observation made by L. Karp that the Schwarz potential of four-dimensional, 
axially-symmetric surfaces can be calculated exactly \cite{Karp}.  
We give some explicit examples and also describe some examples of elliptic growth.
In Section 5, we use $\mathbb{C}^n$ techniques to understand the Schwarz potential's singularity set for a nontrivial example in $\mathbb{R}^n$ including the important case $n=3$.
In Section 6, we discuss the connection to quadrature domains and Richardson's Theorem.

\section{Dynamics of Singularities}

\subsection{The Schwarz Potential}

Suppose $\Gamma$ is a non-singular, real-analytic curve in the plane.  
Then the Schwarz function $S(z)$ is the function that is complex-analytic in a neighborhood of $\Gamma$ and coincides with $\bar{z}$ on $\Gamma$ (see \cite{Davis} 
for a full exposition).  If $\Gamma$ is given algebraically as the zero set of a polynomial $P(x,y)$, 
we can obtain $S(z)$ by making the complex-linear change of variables $z=x+iy$, $\bar{z}=x-iy$, 
and then solving for $\bar{z}$ in the equation $P(\frac{z+\bar{z}}{2},\frac{z-\bar{z}}{2i})=0$.  
For instance, suppose $\Gamma$ is the curve given algebraically by the solution set of the equation $(x^2+y^2)^2=a^2(x^2+y^2)+4\varepsilon^2 x^2$ (``C. Neumann's oval'').  
Then changing variables we have $(z \bar{z})^2=a^2(z \bar{z})+\varepsilon^2 (z+\bar{z})^2$.  
Solving for $\bar{z}$ gives $S(z)=\frac{z(a^2+2\varepsilon^2)+z\sqrt{4a^4+4a^2\varepsilon^2+4\varepsilon^2z^2}}{2(z^2-\varepsilon^2)}$.

Suppose $\Gamma$ is more generally a nonsingular, analytic \emph{hypersurface} in $\mathbb{R}^n$, and consider the following Cauchy problem posed in the vicinity of $\Gamma$.
The solution exists and is unique by the Cauchy-Kovalevskaya Theorem.

\begin{equation}\label{CP}
\left\{
\begin{array}{l}
\Delta w = 0 $ near $ \Gamma \\
w|_{\Gamma} =\frac{1}{2}||{\bf x}||^2 \\
\nabla w|_{\Gamma}={\bf x}
\end{array}\right.
\end{equation}

\begin{definition} The solution $w({\bf x})$ of the Cauchy problem \ref{CP} is called the Schwarz Potential of $\Gamma$.  
\end{definition}

{\bf Example:}  Let $\Gamma := \{{\bf x} \in \mathbb{R}^n : ||{\bf x}||^2 = r^2 \}$ be a sphere of radius $r$. 
When $n=2$, it is easy to verify that $w(z) = r^2 \left( \log|z| + 1/2 - \log(r) \right)$ solves the Cauchy Problem (\ref{CP}),
and in higher dimensions the Schwarz potential is $w({\bf x}) = -\frac{r^n}{(n-2)||{\bf x}||^{n-2}} + \frac{n}{2(n-2)} r^2$.

In $\mathbb{R}^2$, the Schwarz function can be directly recovered from the Schwarz potential.
Consider $S(z) = 2\partial_z w = w_x - iw_y$. 
The Cauchy-Riemann equations for $S$ follow from harmonicity of $w$, 
and $\nabla w = {\bf x}$ on $\Gamma$ implies $S(z) = \bar{z}$ on $\Gamma$.

This gives a partial generalization of the Schwarz function.  
The reflection principle associated with the Schwarz function does not generalize to higher dimensions by this or any other means,
but the Schwarz potential retains other desirable properties.  
In particular, it allows us to generalize Theorem \ref{SF} to higher dimensions.

\subsection{Laplacian growth and the Schwarz potential}\label{sec:LGSP}

The following theorem generalizes Thereom \ref{SF}.  
We consider a family of domains $\Omega_t \subset \mathbb{R} ^n$ with analytic boundary that also has analytic time-dependence.  
Such regularity assumption is natural for us since we are in pursuit of explicit, exact solutions. 
However, we should mention that analyticity of the boundary is a necessary condition for existence of a classical solution, and 
moreover for an analytic initial boundary there exists a unique solution remaining analytic 
with analytic time-dependence for at least some interval of time (see \cite{Escher} and \cite{Tian}).  
Let $w({\bf x},t)$ denote the Schwarz potential of the boundary $\Gamma_t$ of $\Omega_t$.

\begin{thm}[Dynamics of Singularities: $\mathbb{R}^n$]\label{LGSP}
If $\Omega_t$ and $w({\bf x},t)$ are as above then $\Omega_t$ solves the Laplacian growth problem (\ref{LG}) if and only if 
\begin{equation}\label{DS}
\frac{\partial}{\partial t}w({\bf x},t) =-nP({\bf x},t)
\end{equation}
where $n$ is the spatial dimension.  
In particular, singularities of the Schwarz potential in the ``oil domain'' do not depend on time, 
except for one stationed at the source (sink) which does not move but simply changes strength.
\end{thm}

\emph{Remark 1:} This relates the solution of a ``mathematically-posed'' Cauchy problem to that of a ``physically-posed'' Dirichlet problem.

\emph{Remark 2:} Considering the relationship between the Schwarz potential and Schwarz function, in the case of $n=2$, the Theorem says that
$S_t = \frac{\pd}{\pd t}(2\pd_z w) = -4\pd_z P$ which is the content of Theorem \ref{SF}.

\emph{Remark 3:} This is closely related to the celebrated Richardson's Theorem \cite{Rich}. 
Actually, the connection can be established through the role that the Schwarz potential plays in the theory of quadrature domains (see Section \ref{sec:QD}).
Here we are able to give a more elementary proof consisting of two applications of the chain rule.

\begin{proof}
Assume $\{\Omega_t\}$ solves the Laplacian growth problem.  We will show that for each $t$, $w_t({\bf x},t)$ and $-nP({\bf x},t)$ solve the same Cauchy problem.  
Then by the uniqueness part of the Cauchy-Kovalevskaya Theorem, they are identical.

First, we will show that $w_t({\bf x},t)|_{\Gamma_t} = 0$.  Consider a point ${\bf x}(t)$ which is on $\Gamma_t$ at time $t$.  The chain rule gives
\begin{equation}
	\frac{d}{dt}w({\bf x}(t),t)=\nabla w({\bf x}(t),t)\cdot{\bf \dot{x}}(t)+ w_t ({\bf x}(t),t)
	\label{first}
\end{equation}

On the other hand, by the first piece of Cauchy data in (\ref{CP}), 
\begin{equation}\label{cauchy1}
  \frac{d}{dt} w({\bf x}(t),t) = \frac{d}{dt} (\frac{1}{2} ||{\bf x}(t)||^2)={\bf x}(t) \cdot {\bf \dot{x}}(t)
\end{equation}
By the second piece of Cauchy data, 
\begin{equation}\label{cauchy2}
  {\bf x}(t) \cdot {\bf \dot{x}}(t) = \nabla w ({\bf x}(t),t) \cdot {\bf \dot{x}}(t)
\end{equation}
Combining (\ref{cauchy1}) and (\ref{cauchy2}) with equation (\ref{first}) gives
\begin{equation}
	w_t ({\bf x},t)|_{\Gamma_t} = 0
	\label{wdot}
\end{equation}
We are done if we can show that $\nabla w_t |_{\Gamma_t} = -n \nabla p$.  
Given some position, ${\bf x}$, let $T({\bf x})$ assign the value of time precisely when the boundary, $\Gamma_{t}$, of the growing domain passes ${\bf x}$.  
Then by the Cauchy data defining the Scwharz potential (\ref{CP}), $w_{x_k}(x_1,x_2,...,x_n,T(x_1,x_2,...,x_n)) = x_k$.  
Taking the partial with respect to $x_k$ of the $kth$ equation gives $w_{tx_k} T_{x_k} + w_{x_k x_k}=1$.  
Summing these $k$ equations together gives
\begin{equation}
	\nabla w_t \cdot \nabla T + \Delta w = n.
	\label{gradt}
\end{equation}
Since $\Gamma_t$ is the level curve $T({\bf x}) = t$, $\nabla T$ is orthogonal to $\Gamma_t$, and $\nabla T = \frac{v_n}{||v_n||^2}$, where 
$v_n$ is the normal velocity of $\Gamma_t$.  Recall, $v_n=-\nabla P$.  Thus, $\nabla T = \frac{-\nabla P}{||\nabla P||^2} $.  
Substitution into equation (\ref{gradt}) gives $\nabla w_t \cdot \nabla P = -n||\nabla P||^2$.

By equation (\ref{wdot}), $w_t|_{\Gamma_t}=0$, which implies that $\nabla w_t$ and $\nabla P$ are parallel.  
So, $\nabla w_t|_{\Gamma_t} = -n\nabla P$.
\end{proof}

\subsection{A Cauchy problem connected to Elliptic Growth}\label{sec:EG}

A natural generalization of the Laplacian growth problem is to allow a non-constant ``filtration coefficient'' $\lambda$ and  ``porosity'' $\rho$.
Then instead of Laplace's equation the pressure satisfies $\text{div}( \lambda \rho \nabla P) = 0 $
and should have a singularity at the sink of the same type as the fundamental solution to this elliptic equation.
Moreover, the Darcy's law determining the boundary velocity becomes $v_n = -\lambda \nabla P$.
For details, see \cite{KMP}, \cite{L}, \cite{LY}.
Physically, this models the problem in a non-homogeneous medium 
and also relates to the case of Hele-Shaw cells on curved surfaces (in the absence of gravity) studied in \cite[Ch. 7]{Etingof}.

We can formulate an equation similar to equation (\ref{DS}) that relates the pressure function of an elliptic growth 
to the time-dependence of the solution to a certain Cauchy problem. 
Let $q({\bf x})$ be a solution of the Poisson equation, 
\begin{equation}\label{Q}
 \text{div} (\lambda \rho \nabla q) = n \rho,
\end{equation}
where $n$ is the spatial dimension.
Recall that a solution $q$ can be obtained by taking the convolution of $\rho$ with the fundamental solution of the homogeneous elliptic equation (if one exists).
We associate with an elliptic growth having filtration $\lambda$ and porosity $\rho$, the solution $u$ of the following Cauchy problem.

\begin{equation}\label{CPEG}
\left\{
\begin{array}{l}
\text{div}\left(\lambda \rho \nabla u \right)  = 0 $ near $ \Gamma \\
u|_{\Gamma} = q \\
\nabla u|_{\Gamma}= \nabla q
\end{array}\right.
\end{equation}

We can think of $u$ a ``generalized Schwarz potential''.
We have the following direct generalization of Theorem \ref{LGSP}.
As in Section \ref{sec:LGSP}, assume $\Omega_t$ has analytic boundary with analytic time-dependence.

\begin{thm}\label{EGSP}
If $\Gamma_t = \partial \Omega_t$ and $u({\bf x},t)$ is the solution to \ref{CPEG} posed on $\Gamma_t$ then $\Omega_t$ is an elliptic growth with pressure function $P({\bf x},t)$ if and only if 
\begin{equation}\label{EG}
\frac{\partial}{\partial t}u({\bf x},t) =-n P({\bf x},t)
\end{equation}
\end{thm}

\begin{proof}
As in the proof of Theorem \ref{LGSP}, we show that both sides of \ref{EG} solve the same Cauchy problem.

The first part of the argument is similar in showing that.  
\begin{equation}\label{udot2}
	u_t ({\bf x},t)|_{\Gamma_t} = 0
\end{equation}
Consider a point ${\bf x}(t)$ which is on $\Gamma_t$ at time $t$.  The chain rule gives
\begin{equation}
	\frac{d}{dt}u({\bf x}(t),t)=\nabla u({\bf x}(t),t)\cdot{\bf \dot{x}}(t)+ u_t ({\bf x}(t),t)
	\label{first2}
\end{equation}

On the other hand, by the first piece of Cauchy data in (\ref{CPEG}) and the chain rule again,
\begin{equation}\label{diffq}
  \frac{d}{dt} u({\bf x}(t),t) = \frac{d}{dt}  q({\bf x}(t))=\nabla q({\bf x}(t)) \cdot {\bf \dot{x}}(t)
\end{equation}
By the second piece of Cauchy data, 
\begin{equation}\label{qcauchy}
  \nabla q({\bf x}(t)) \cdot {\bf \dot{x}}(t) = \nabla u ({\bf x}(t),t) \cdot {\bf \dot{x}}(t)
\end{equation}
Combining (\ref{diffq}) and (\ref{qcauchy}) with equation (\ref{first2}) gives the equation (\ref{udot2}).

We are done if we can show that $\nabla  u_t  |_{\Gamma_t} = - n \nabla P$.  

We again let $T({\bf x})$ assign the value of time when $\Gamma_{t}$ passes ${\bf x}$.  
Then by the Cauchy data defining $u$, $\nabla u({\bf x},T({\bf x})) = \nabla q$.
Multiply both sides by $\lambda \rho$ and take the divergence:
\begin{equation}
  \lambda \rho \nabla u_t \cdot \nabla T + \text{div}(\rho \lambda \nabla u) = \text{div}({\rho \lambda \nabla q}),
\end{equation}
which, by definition of $u$ and $q$, simplifies to
\begin{equation}\label{gradt2}
  \lambda \nabla u_t \cdot \nabla T = n.
\end{equation}
As in the proof of Theorem \ref{LGSP}, $\nabla T = \frac{v_n}{||v_n||^2}$, where 
$v_n$ is the normal velocity of $\Gamma_t$. Except now $v_n=-\lambda \nabla P$.  Thus, $\nabla T = \frac{-\nabla P}{\lambda||\nabla P||^2} $.  
Substitution into equation (\ref{gradt2}) gives $\nabla u_t \cdot \nabla P = -n||\nabla P||^2$.

By equation (\ref{udot2}), $u_t|_{\Gamma_t}=0$, which implies that $\nabla u_t$ and $\nabla P$ are parallel.  
So, $\nabla u_t|_{\Gamma_t} = -n\nabla P$.
\end{proof}

Let us discuss a special case of the above.
Suppose that the Problem (\ref{Q}) has a solution $q$ that is entire. 
Let $\alpha$ denote $\lambda \rho$, and suppose $\alpha$ is also entire.
For instance, $\lambda = \frac{1}{x^2+1}$ and $\rho = x^2 + 1$ gives $\alpha = 1$ and $q = x^4 + x^2$.
When $\alpha=1$ as in this case, the ``elliptic growth'' is just a Laplacian growth with a variable-coefficient law governing the boundary velocity.
The problem (\ref{CPEG}) defining $u$ becomes a Cauchy problem for Laplace's equation with entire data.
This is the realm of the Schwarz potential conjecture formulated by Khavinson and Shapiro:

\begin{conj}[Khavinson, Shapiro]
 Suppose $u$ solves the Cauchy problem for Laplace's equation posed on a nonsingular analytic surface $\Gamma$ with real-entire data.
Then the singularity set of $u$ is contained in the singularity set of the Schwarz potential $w$.
\end{conj}

The conjecture holds in the plane and has been shown to hold ``generically'' in higher dimensions \cite{Stern}.
If the conjecture is true, then for the case when $\alpha = 1$, the singularities of $u$ are controlled throughout time by those of $w$.
Combining this with Theorems \ref{LGSP} and \ref{EGSP} implies that, given a solution of the Laplacian growth problem (\ref{LG}),
the exact same evolution can be generated amid an elliptic growth law with $\alpha = 1$ by a pressure function having singularities at the same locations as those of $w$.
The singularities may have different time-dependence and be of a different type.

For instance, consider the plane and the simplest Laplacian growth of suction from the center of a circle 
so that at time $t$, the Schwarz function is $\frac{1-t}{z}$ (a constant rate of suction).
Let us determine the pressure required to generate the same process when $\lambda = \frac{1}{x^2+1}$ and $\rho = x^2+1$.
To solve for $u$, we notice that $\partial_z u$ is analytic and coincides with $2x^3 + x$ on the shrinking circle.
Since $x = \frac{z+S(z)}{2}$ on the boundary, we have $\partial_z u = (z+\frac{1-t}{z})^3/4+z/2+\frac{1-t}{2z}$.
This is even true off the boundary since both sides are analytic.
The singular terms are $\frac{5-8t+3t^2}{4z}$ and $\frac{1-3t + 3t^2-t^3}{4z^3}$.
Thus, in order to generate the same ``movie'', the pressure must have a fundamental solution type singularity along with a weak ``multi-pole'' at the origin, 
both diminishing at non-constant rates.

\section{Examples}

In this section we will understand some explicit solutions in terms of the Theorems \ref{SF}, \ref{LGSP}, and \ref{EGSP}.

\subsection{Laplacian growth two dimensions}\label{sec:plane}

First we review some familiar examples in the plane, where
typically a time-dependent conformal map is introduced.  
Instead, we work entirely with the Schwarz function and check that Theorem \ref{SF} is satisfied.

{\bf Example 1:} Consider the family of domains $D$ with boundary given by the curves $\{z:z=aw^2+bw,|w|<1 \}$ with $a, b$ real.  
The Schwarz function is given by $S(z)= -2ab/(a-\sqrt{a^2+4bz})+4b^3/(a-\sqrt{a^2+4bz})^2$ which has a single-valued branch in the interior of the curve 
for appropriate parameter values $a$ and $b$.  
The only singularities of the Schwarz function interior to the curve are a simple pole and a pole of order two at the origin.  
Given an initial domain from this family we can choose a one-parameter slice of domains so that the simple pole increases (resp. decreases) while the pole of order two does not change.  This gives an exact solution to the Laplacian growth problem with injection (resp. suction) taking place at the origin.  In the case of injection, the domain approaches a circle.  In the case of suction, the domain develops a cusp in finite time.

Instead of just one sink or source ${\bf x_0}$ with rate $Q$, 
let us extend problem \ref{LG} by allowing for multiple sinks and/or sources ${\bf x_i}$ with suction/injection rates $Q_i$.  
This is the formulation of the problem which is often made, for instance, see the excellent exposition \cite{Etingof}.  
The proof of Theorem \ref{LGSP} carries through without changes so that the time-derivative of the Schwarz potential still coincides with $-nP({\bf x},t)$.
The only difference is that now there can be multiple time-dependent point-singularities inside $\Omega_t$.

{\bf Example 2:} We first consider the family of curves mentioned in Section 2.1.  
The Schwarz function of the boundary is $$S(z)=\frac{z(a^2+2\varepsilon^2)+z\sqrt{4a^4+4a^2\varepsilon^2+4\varepsilon^2z^2}}{2(z^2-\varepsilon^2)}$$
which has two simple poles at $z= \pm \varepsilon$ each with residue $(a^2+2\varepsilon^2)/2$.  
In order to satisfy the conditions imposed by Theorem \ref{LGSP} we choose $\varepsilon=1$ constant.  
Then we choose $a(t)$ to be decreasing (increasing) to obtain suction (injection) at two sinks (sources).  
In the case of suction, the oval forms an indentation at the top and bottom and becomes increasingly pinched 
as the boundary approaches two tangent circles centered at $\pm 1$, the positions of the sinks (see Figure \ref{fig:Neumann}).  

\begin{figure}[t]
    \includegraphics[scale=.5]{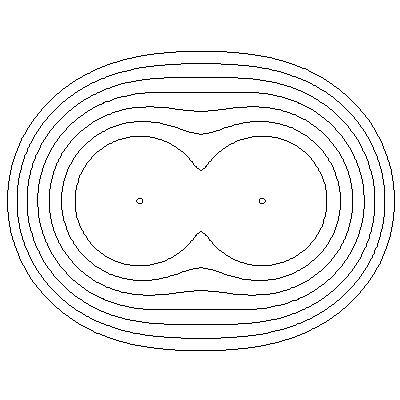}
    \caption{An example with two sinks.}
    \label{fig:Neumann}
\end{figure}

For the next example, we consider the case of Problem (\ref{LG}) where the ``oil domain'' $\Omega_t$ is unbounded with a sink at infinity.

{\bf Example 3:} We recall the Schwarz function for an ellipse given by the solution set of the equation $\frac{x^2}{a^2}+\frac{y^2}{b^2}=1$.
Changing variables we have $\frac{(z+\bar{z})^2}{a^2}-\frac{(z-\bar{z})^2}{b^2}=4$.
Solving for $\bar{z}$ gives $S(z) = \frac{a^2 + b^2}{a^2 - b^2}z+\frac{2ab}{b^2 - a^2}\sqrt{z^2 + b^2 - a^2}$.
$S(z)$ has a square root branch cut along the segment joining the foci $\pm \sqrt{a^2-b^2}$, but 
we are only interested in the exterior of the ellipse, where $S(z)$ is free of singularities.
This already guarantees that any evolution of ellipses that has analytic time dependence can be generated 
by preparing the correct asymptotic pressure conditions to match $S_t(z,t)$ which is only singular at infinity.
In other words, we can use equation \ref{DSP} to work backwards in specifying the pressure condition to generate the given flow.
Since there are no finite singularities, we only have to specify the conditions at infinity.
A realistic case is if the asymptotic condition is steady and isotropic:
$S_t(z \rightarrow \infty,t) \approx k/z$ for a constant $k$ independent of $t$.
Take a homothetic growth with $a(t) = a \sqrt{t}$ and $b(t) = b \sqrt{t}$ from some initial ellipse with semi-axes $a$ and $b$. 
Then $S(z,t) = \frac{a^2 + b^2}{a^2 - b^2}z+\frac{2ab}{b^2 - a^2}\sqrt{z^2 + t(b^2 - a^2)}$,
and we have $S_t(z,t) = k \frac{1}{\sqrt{z^2+t(b^2-a^2)}}$, where $k=2ab$.

\subsection{Examples and non-examples in $\mathbb{R}^4$}\label{sec:r4}

Next, we consider axially-symmetric, four-dimensional domains.  
This turns out to be simpler than the more physically relevant $\mathbb{R}^3$, 
and we will see in the next subsection that it is
is equivalent to certain cases of elliptic growth in two and three dimensions.
Lavi Karp \cite{Karp} has given a procedure, including several explicit examples, 
for obtaining the singularities of the Schwarz potential for a domain $\Omega$ that is the rotation into $\mathbb{R}^4$ 
of a domain in $\mathbb{R}^2$ with Schwarz function $S(z)$.  
We outline here this procedure for finding the Schwarz potential $w(x_1,x_2,x_3,x_4)$.  
Since $w$ solves a Cauchy problem for axially symmetric data posed on an axially symmetric hypersurface, with, say $x_1$ as the axis of symmetry, 
it is a function of two variables.  Write $x=x_1, y=\sqrt{x_2^2+x_3^2+x_4^2}$, and $w(x_1,x_2,x_3,x_4)=U(x,y)$.  
What makes $\mathbb{R}^4$ convenient to work with is the fact that $V(x,y) = y \cdot U(x,y)$ is a harmonic function of the variables $x$ and $y$.  
Thus, finding $U(x,y)$ is reduced to solving an algebraic Cauchy problem in the plane, which can be done in terms of the Schwarz function $S(x+iy)=S(z)$.  
The steps for writing this solution are outlined below.

{\bf Step 1:} Write $f(z)=\frac{i}{2}S(z) \cdot (S(z)-2z)$.
\\{\bf Step 2:} Find a primitive function $F(z)$ for $f(z)$.
\\{\bf Step 3:} Write $V(x,y)=Re\{F(z)\}$.  Then the Schwarz potential for $\Omega$ is $U(x,y)=\frac{V(x,y)+const.}{y}$.

One of the examples carried through this procedure in \cite{Karp} is the family of ``limacons'' from Example $1$.  
The result is that the Schwarz potential can be expanded about the origin as 
$w(x_1,x_2,x_3,x_4)=A_2(a,b) \left( \frac{\partial}{\partial x_1}\right)^2|{\bf x}|^{-2}+A_1(a,b)\frac{\partial}{\partial x_1}|{\bf x}|^{-2}+A_0(a,b)|{\bf x}|^{-2}+H({\bf x})$, 
where $H({\bf x})$ is harmonic and $A_2(a,b)= -b^2a^4/12$, $A_1(a,b)= ba^2(a^2+2b^2)/2$, and $A_0(a,b)= -(a^4+6a^2b^2+2b^4)/2$.  
We can interpret one-parameter slices of this family as a Laplacian growth if we further extend problem \ref{LG} to allow for ``multi-poles''
(see \cite{EEK} for a discussion of multi-pole sollutions in the plane).
If we want a Laplacian growth with just a simple sink then according to the dynamics-of-singularities imposed by Theorem \ref{LGSP},  
we need to choose the time-dependence of $a$ and $b$ so that the only 
singularity whose coefficient changes is the fundamental-solution type singularity $A_0(a,b)|{\bf x}|^{-2}$.
Thus, where $C_1$, $C_2$ are constants, we need to have:
\begin{equation}\label{overdet}
\left\{
\begin{array}{l}
A_2(a(t),b(t)) = C_2 \\
A_1(a(t),b(t)) = C_1 \\
\end{array}\right.
\end{equation}
Unfortunately, solutions $a$ and $b$ of this system are locally constant so that $A_0$ must then be constant and the whole surface does not move at all.  
The other examples of axially symmetric domains considered in \cite{Karp} also require introducing multi-poles or even a continuum of singularities, 
otherwise the conditions imposed by simple sources/sinks leads to a similarly overdetermined system.  
Roughly speaking, the difficulty is that $f(z)$ from Step 1 above generally has more singularities than $S(z)$.  
Thus, if a class of domains in the plane has enough parameters to control the singularities and obtain a Laplacian growth,
then rotation into $\mathbb{R}^4$ introduces more singularities which must be controlled with the same number of parameters.

There are exceptions: we can describe some exact solutions involving a simple source and sink (no multipoles).  
Consider the hypersurfaces of revolution obtained by rotation from a family of curves whose Schwarz functions have two simple poles at $z= \pm 1$
(with not necessarily equal residues).  
This is a two-parameter family of surfaces; 
as parameters, we can take the residues of the Schwarz function of the profile curves.
Let $\Omega$ denote the domain in the plane bounded by the profile curve.
The Schwarz function has the form 
$$S(z) = \frac{A}{z-1}+\frac{B}{z+1}+c(A,B)+d(A,B)z+z^2 H(z,A,B),$$ 
where $H(z,A,B)$ is analytic in $\Omega$.
In what follows, we will suppress the dependence on $A$ and $B$ of higher-degree coefficients.  Following Steps 1 through 3 above, we have 
$$f(z) = \frac{i}{2}S(z) \cdot (S(z)-2z)$$
$$= \frac{i}{2}\left( \frac{A^2}{(z-1)^2}+\frac{B^2}{(z+1)^2}+ \frac{C_1}{z-1} - \frac{C_2}{z+1} + H_1(z) \right),$$
where $H_1(z)$ is analytic in $\Omega$.
Then for Step 2 we need a primitive function for $f(z)$ which is 
$$F(z) = \frac{i}{2}\left( \frac{A^2}{(z-1)}+\frac{B^2}{(z+1)}+ C_1 \text{Log}(z-1) - C_2 \text{Log}(z+1) + H_2(z) \right),$$ 
where $H_2(z)$ is analytic in $\Omega$.

Then for Step 3 we have $V(x,y)=Re\{F(z)\} = \frac{A^2 y}{(x-1)^2+y^2}+\frac{B^2 y}{(x+1)^2+y^2} + C_1 \arg(z-1) + C_2 \arg(z+1) + H_3(z)$.

If we can vary $A$ and $B$ in a way that keeps $C_1$ and $C_2$ constant, 
then the time-derivative of the Schwarz function will satisfy the dynamics-of-singularities condition.  
This seems at first to be another overdetermined problem, but actually $C_1$ and $C_2$ must be equal!  
Otherwise, the two branch cuts of $C_1 \arg(z-1)$ and $C_2 \arg(z+1)$ will not cancel eachother outside the interval $[-1,1]$, 
and the Schwarz potential will become singular on the surface itself.  
This cannot happen since the surface has no points (in $\mathbb{R}^n$) that are characteristic for the Cauchy problem.
Therefore, since $C_1$ and $C_2$ are equal, we spend only one dimension of our parameter space controlling the ``non-physical'' segment of singularities.
This leaves freedom for the ``physical'' singularities to move, at least locally, along a one-dimensional submanifold of parameters.
Figure \ref{r4} shows the evolution of the profile curve for a typical example that can be obtained in this way.

We omit the cumbersome formulae for the time-dependence of coefficients in the algebraic description of such exact solutions.
The two-parameter family of hypersurfaces from which they are selected can be described by the solution set of:
$$\frac{\left((x_1+\frac{h}{2(a^2-h^2)})-h \cdot ((x_1+\frac{h}{2(a^2-h^2)})^2+x_2^2+x_3^2+x_4^2)\right)^2}{a^2}$$
$$-\frac{(4(a^2-h^2)^2-a^2)(x_2^2+x_3^2+x_4^2)}{a^2(a^2-h^2)}=\left(\left(x_1+\frac{h}{2(a^2-h^2)}\right)^2+x_2^2+x_3^2+x_4^2\right)^2.$$   

\begin{figure}[t]
    \includegraphics[scale=.4]{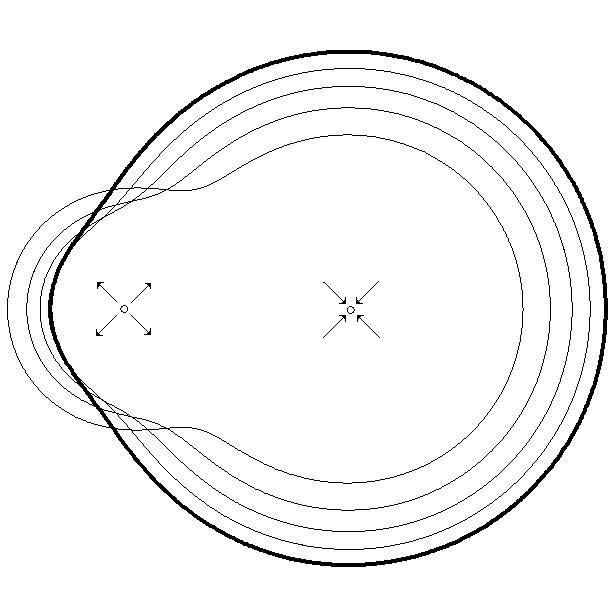}
    \caption{A profile of an axially-symmetric solution in $\mathbb{R}^4$ with injection at one point and suction at another. The initial curve is plotted in bold.}
    \label{r4}
\end{figure}

Similarly, one can obtain examples where the Schwarz function has three or more simple poles.
Again the suction/injection rates will have to occur in a prescribed way 
or else the time-derivative of the Schwarz potential will have singular segments which are difficult to interpret physically.

\emph{Remark:} The rigidity of the inter-dependence of injection/suction rates in the above example is 
made less severe by the fact that the initial and final domains 
only depend on the total quantities injected and removed at the source and sink respectively,
and they are independent of the rates and order of work of the source and sink
(see \cite{Etingof}: the proof extends word for word to higher dimensions).
Thus, injection and suction can happen in any manner, say one at a time, and we will lose the ``movie'' but retain the final domain.

In the next section we will be interested in examples that correspond to axially symmetric surfaces that do not intersect the axis of symmetry.
For instance, to generate a torus, we can choose the profile curve to be a circle of radius $R$ and center $ai$, $a > R > 0$.
The Schwarz function is $S(z) = \frac{R^2}{z-ai} - ai$.
Step 1 gives $f(z) = i/2 (\frac{R^2}{z-ai} - ai)(\frac{R^2}{z-ai} - ai-2z)$.
Step 2 gives $F(z) = \frac{-iR^4}{2(z-ai)} + 2R^2aLog(z-ai) + H(z)$, where $H(z)$ is analytic.
Step 3 gives $V(z) = \frac{-R^4(y-a)}{2|z-ai|^2} + 2R^2a\log|z-ai| + R(z)$, where $R(z)$ is free of singularities.
Finally, the singular part of $U(x,y)$ is $\frac{R^4}{2a y}\frac{\partial}{\partial y} \frac{1}{x^2+y^2}+\frac{2R^2}{y}\log|z-ai|$.

This calculation for the Schwarz potential of the four-dimensional torus was carried out in \cite{Aberra}
and discussed in connection with a classical mean-value-property for polyharmonic functions.

\subsection{Examples of elliptic growth}\label{EGex}

Examples of axially-symmetric, four-dimensional Laplacian growth also solve certain elliptic growth problems in two and three dimensions.
The two-dimensional profile solves the planar elliptic growth problem where the filtration coefficient $\lambda = 1$ is constant, and the porosity $\rho(x,y) = y^2$.
Indeed, we can check that Theorem \ref{EGSP} is satisfied.
The Schwarz potential $U(x,y)$, reduced to two variables, satisfies the equation $\Delta U + \frac{2 U_y}{y} = 0$.
Since $\text{div}\left(y^2 \nabla U \right) = y^2 \Delta U + 2y U_y$, then $U$ solves the Cauchy problem

\begin{equation}
\left\{
\begin{array}{l}
\text{div}\left(y^2 \nabla U \right)  = 0 $ near $ \Gamma \\
U|_{\Gamma} = q \\
\nabla U|_{\Gamma}= \nabla q
\end{array}\right.
\end{equation}

with $q(x,y) = (x^2 + y^2)/8$ solving the Poisson equation $\text{div}\left(y^2 \nabla q \right)  = y^2$.

The three dimensional surfaces of revolution generated by the same profile curves solve a three-dimensional elliptic growth 
if we choose $\lambda=1$ again constant and porosity $\rho(x,y,z) = \sqrt{y^2 + z^2}$.

It is most interesting when the domain, at least initially, avoids the line $\{y=0\}$ where $\rho(x,y)$ vanishes.
Consider, for instance, a circle of radius $R$ centered at $ai$.
This corresponds to the calculation at the end of Section \ref{sec:r4} for the four-dimensional torus.
Accordingly, a shrinking circle can be generated by a simple source combined with a ``dipole flow'' positioned at the center of the circle.

A similar calculation applies more generally when $\lambda(x,y) = y^{2-m}$, $\rho(x,y) = y^{m}$, with $m$ a positive integer, and we can consider more general domains than circles.
For instance, a well-known classical solution of the Laplacian growth in the plane involves domains $\Omega_t$ conformally mapped from the unit disc by polynomials.
Physically the solution has a single sink positioned at the image of the origin under the conformal map.
The Schwarz function of such an $\Omega_t$ is meromorphic except at the sink where its highest order pole coincides with the degree of the polynomial.
So, $S(z) = \sum_{i=1}^{k}{\frac{a_i}{z^i}} + H(z)$, where $H(z)$ is analytic in $\Omega_t$.
The solution $q$ of $\text{div}\left(y^2 \nabla q \right)  = y^m$ is $q(x,y) = \frac{y^{m+2}}{(m+2)(m+3)}$.
To solve for $U(x,y)$ we first notice that $V(x,y) = yU(x,y)$ is harmonic and solves a Cauchy problem with data $yq(x,y) = \frac{y^{m+3}}{(m+2)(m+3)}$.
Thus, $\partial_z V = -\frac{i}{2}\frac{y^{m+2}}{(m+2)}=-\frac{i}{2}\frac{(z-S(z))^{m+2}}{(m+2)(2i)^{m+2}}$ can be analytically continued away from the boundary.
As a result, the flow can be generated by a combination of ``multipoles'' positioned at the same point of order not exceeding $k(m+2)$.
This resembles the result of I. Loutsenko \cite{L} stating that 
the same evolution can be generated by multipoles of a certain order 
under an elliptic growth where $\rho=1$ constant and $\lambda = \frac{1}{y^{2p}}$, with $p$ a positive integer.

\begin{figure}[b]
    \includegraphics[scale=.5]{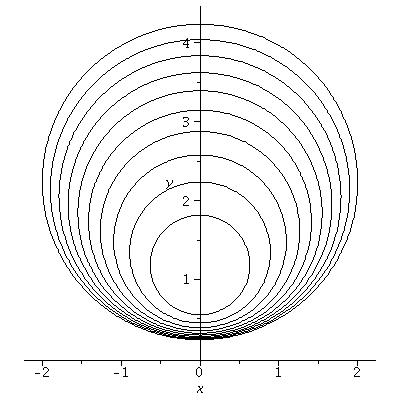}
    \caption{An elliptic growth with multi-poles of order up to 3 positioned at $z=i$.  The Schwarz function has a moving singularity.}
    \label{fig:pole}
\end{figure}

This fails, in an interesting way, for negative values of $m$.  For instance, if $\lambda(x,y) = y^4$ and $\rho(x,y) = 1/y^2$, 
then a circle of shrinking radius $R$ centered at $ai$ is not generated by multipoles positioned at $ai$.
Instead, the generalized Schwarz potential $U(x,y)$ has singularities at the moving point $i \sqrt{a^2-r^2}$.
If we instead allow the center of the shrinking circle to move in a way that keeps $\sqrt{a^2-r^2}$ constant, 
then the evolution can be generated by multi-poles at this point of order up to $3$ (see figure \ref{fig:pole}).
To reiterate, for this evolution of shrinking circles with moving center, the generalized (elliptic) Schwarz potential 
is singular at a stationary point while the analytic Schwarz function has a moving singularity.
Such an example has been anticipated in \cite{KMP}, 
where a system of nonlinear ODEs was given governing both the strength and the moving position of the Schwarz function's singularities under an elliptic growth.

\section{The Schwarz potential in $\mathbb{C}^n$}\label{sec:Cn}

The previous sections call for a deeper look into the singularities that can arise from Cauchy's problem for the Laplace equation.
Certain techniques can only be applied if the problem is ``complexified''.
According to the algebraic form of the initial surface and data, we can allow each variable to assume complex values.
We then consider the Cauchy problem in $\mathbb{C}^n$ where the original, physical problem becomes a relatively small slice.
We can loosely describe the advantage of a $\mathbb{C}^n$-viewpoint as follows:

Consider first the wave equation in $\mathbb{R}^n$.  
If the initial surface is non-singular and algebraic and the data is real-entire, then where can the solution have singularities?
A singularity can propagate to some point if the backwards light-cone from this point is tangent to the initial surface.
The same is true, at least heuristically, for the Laplace equation, 
except the ``light cone'' emanating from a point ${\bf x_0}$ is the \emph{isotropic cone}
$:=\{ \sum_{i=1}^n{(z_i-x^0_i)^2}=0 \}$, residing in $\mathbb{C}^n$ 
and only touching $\mathbb{R}^n$ at ${\bf x_0}$.
Thus, the initial source of the singularity is located on the part of the complexified surface only visible if the problem is lifted to $\mathbb{C}^n$.

``Leray's principle'' gives the general, precise statement of the above description of propagation of singularities.
It is only known to be rigorously true in a neighborhood of the initial surface.
In two dimensions, where the Schwarz potential can be calculated easily, one can check examples to see if Leray's principle gives correct global results (it seems to).
At the same time, this gives an appealing geometric ``explanation'' of the source of singularities 
and reveals that they are the ``foci'' of the curve in the sense of Pl\"ucker (see \cite[Section 1]{Johnsson} and the references therein).

In arbitrary dimensions, G. Johnsson has given a global proof \cite{Johnsson} of Leray's principle for quadratic surfaces.
As Johnsson points out, a major step in the proof relies on the fact that the gradient of a quadratic polynomial is linear, 
so that a certain system of equations can be inverted easily. 
This becomes much more difficult for surfaces of higher degree, indeed, perhaps prohibitively difficult even for specific examples.

In this section we consider a family of surfaces of degree four, 
the surfaces of revolution generated by the Neumann ovals from Example 2 in Section \ref{sec:plane}.
Leray's principle gives an appealing geometric ``explanation'' for the singularities of the Schwarz potential in this example,
but for the rigorous proof, we apply an ad hoc combination of other $\mathbb{C}^n$ techniques (actually $\mathbb{C}^2$, after taking into account axial symmetry).

We require the following two local extension Theorems.

\begin{thm}[Zerner]\label{Zerner}
 Let $v$ be a holomorphic solution of the equation $L v = 0$ in a domain $\Omega \subset \mathbb{C}^n$ with $C^1$ boundary,
and assume that the coefficients of $L$ are holomorphic in $\overline{\Omega}$.
Let $z_0 \in \partial \Omega$.
If $\partial \Omega$ is non-characteristic at $z_0$ with respect to $L$ then $v$ extends holomorphically into a neighborhood of $z_0$.
\end{thm}

In order to define \emph{non-characteristic} for a real hypersurface given by the zero set of $\phi$, suppose the polynomial $P({\bf x},\nabla)$ expresses the leading order term of $L$.
Then $\Gamma$ is characteristic at $p$ if $P({\bf x},\nabla \phi)$ vanishes at ${\bf x}=p$.
For instance, if $L$ is the Laplacian then the condition for $\{\phi=0 \}$ to be characteristic is $\sum_{i=1}^n {\phi_{x_i}^2} = 0$.

In the statement of the next theorem, $M$ is a hypersurface (of \emph{real} codimension one) dividing a domain $\Omega$ into $\Omega_+$ and $\Omega_-$, 
and $X$ is a holomorphic hypersurface (of \emph{complex} codimension one).

\begin{thm}[Ebenfelt, Khavinson, Shapiro]\label{EKS}
Assume that $M$ is non-characteristic for $L$ at $p_0 \in M$, 
and that the holomorphic hypersurface X is non-singular at $p_0$ and meets $M$ transversally at that point.
Then any holomorphic solution $v$ in $\Omega_{-}$ of
$P(Z,D)v=0$
extends holomorphically across $p_0$.
\end{thm}

\begin{thm}\label{Neumann}
 Let $W({\bf x})$ be the Schwarz potential of the boundary $\Gamma$ of the domain 
$\Omega:=\{{\bf x} \in \mathbb{R}^n : (\sum_{i=1}^{n}{x_i^2})^2 - a^2\sum_{i=1}^{n}{x_i^2} - 4x_1^2  < 0\}$.
Then $W$ can be analytically continued throughout $\Omega \setminus B$ where $B$ is the segment $\{ x_1 \in [-1,1], x_j = 0 \text{ for } j = 2,..,n \}$. 
\end{thm}

\emph{Remark (i):} In the plane, it is easily seen that $W$ is only singular at the endpoints of the segment (see Example 2 in Section \ref{sec:plane}).
In $\mathbb{R}^4$, it is an example done by L. Karp \cite{Karp}, who showed that
the Schwarz potential has two fundamental solution type singularities at the endpoints along with a uniform jump in the gradient across the segment.

\emph{Remark (ii):} The three-dimensional consequence of this theorem is 
that if we take the surfaces of revolution generated by the Neumann ovals in Example 2 from Section \ref{sec:plane} 
then the resulting evolution is a ``Laplacian growth'' generated by a pressure function 
having some distribution of singularities confined to the segment $\{x \in [-1,1], y=0, z=0 \}$.
This driving mechanism is still rather obscure though, so in the next section we describe an approximation by finitely many simple sinks.

\begin{proof}
We first recall that $W({\bf x})$ is real-analytic in a neighborhood of each nonsingular point of the initial surface (in $\mathbb{R}^n$).
Indeed, if the surface is nonsingular, $\nabla \Phi|_{\Gamma} \neq {\bf 0}$ so that $||\nabla \Phi||^2 |_{\Gamma} \neq 0$ 
so that $\Gamma$ is everywhere non-characteristic (in $\mathbb{R}^n$) for Laplace's equation
and the Cauchy-Kovalevskaya Theorem applies.

Next we write $W(x_1,x_2,...,x_n) = u(x,y)$ where $x=x_1$ and $y=\sqrt{x_2^2+x_3^2+...+x_n^2}$, 
and we recall the axially-symmetric reduction of Laplace's equation: $\Delta u + \frac{(n-2)u_y}{y}=0$. 
Since $u$ solves a Cauchy problem for which the data and boundary are analytic, 
the problem can be lifted to $\mathbb{C}^2$.
So $u(x,y)$ can be viewed as the restriction to $\mathbb{R}^2$ of the solution $u(X,Y)$, 
valid for $X$ and $Y$ each taking complex values.

We make the linear change of variables $X=\frac{z+w}{2}$, $Y=\frac{z-w}{2i}$: 
$u_{z w}+\frac{(n-2)(u_{z}-u_{w})}{z-w}=0$.
Next we make another change of variables $z=f(\xi)$, $w=f(\eta)$, using the conformal map
$$f(\xi)=\frac{(R^4-1) \xi}{R(R^2 - \xi^2)}$$ 
from the unit disk to the profile of $\Omega$ (for appropriate value of $R$) which is Neumann's oval (see Figure \ref{NeumConf}).

\begin{figure}[h]
    \includegraphics[scale=.5]{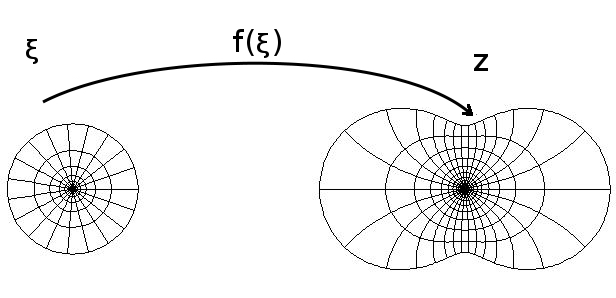}
    \caption{The conformal map from the disc to the Neumann oval.  This simplifies the $\mathbb{C}^2$ geometry but makes the PDE more complicated.}
    \label{NeumConf}
\end{figure}

Write $v(\xi,\eta)= u(f(\xi),f(\eta))$.  Then $v_{\xi} = u_{z}(f(\xi),f(\eta))\cdot f'(\xi)$, 
and the equation satisfied by $v$ is 
$\frac{v_{\xi \eta}}{f'(\xi) f'(\eta)}+\frac{n-2}{f(\xi)-f(\eta)}\left(\frac{v_{\xi}}{f'(\xi)}-\frac{v_{\eta}}{f'(\eta)} \right)=0$,
or $(f(\xi)-f(\eta)) v_{\xi \eta} + (n-2)\left( f'(\eta) v_{\xi} - f'(\xi) v_{\eta}\right)=0$. 
Upon clearing denominators, the leading-order term is
\begin{equation}\label{PDE}
\frac{(R^4-1)}{R}(R^2 - \xi^2)(R^2 - \eta^2) \left( \xi(R^2 - \eta^2) - \eta(R^2 - \xi^2) \right) v_{\xi \eta}.
\end{equation}

After these transformations, we arrive at a Cauchy problem posed on $\{ \xi \eta = 1 \}$,
with data $v = 1/2f(\xi)f(\eta)$, $v_{\xi} = f(\eta)f'(\xi)$, and $v_{\eta} = f(\xi)f'(\eta)$.
According to the form of the leading-order term \ref{PDE}, the characteristic points of $\{ \xi \eta = 1 \}$ are 
$(\pm 1, \pm 1)$, $(\pm R,\pm 1/R)$, $(\pm 1/R,\pm R)$.

The restriction of $v$ to the non-holomorphic set $\eta=\bar{\xi}$ corresponds to the original problem.
Since $W({\bf x})$ was observed to be analytic near the initial surface, 
$v(z,w)$ is analytic in a $\mathbb{C}^2$ neighborhood of the circle $\{\xi \eta = 1, \eta = \bar{\xi} \}$,
even at the characteristic points $(\pm 1, \pm 1)$ on the axis of symmetry.
We analytically continue $v$ from each point on this circle along a radial path toward the origin.
Let $P(\theta) = (e^{i\theta},e^{-i \theta})$.
We consider two cases.
For the first case, $\theta \neq 0$ and $\neq \pi$, and $v$ can be continued up to the origin.
For the second case, when $\theta = 0$ or $= \pi$, the analytic continuation stops at $(1/R,1/R)$ and $(-1/R,-1/R)$ respectively.
Thus, $v$ can be analytically continued to the disk minus the segment joining these two points.
This transforms (by inverting the conformal map) to the statement we are trying to prove about $W$.
For each case we construct a globalizing family in a similar manner to the proof of the Bony-Shapira Theorem \cite{Bony}.

CASE 1: Suppose $\theta \neq 0$ and $\theta \neq \pi$ so that $(e^{i\theta},e^{-i \theta})$ is not on the axis of symmetry.
Let $0<s<1$ be arbitrary.  
We establish the continuability of $v$ to a neighborhood of the segment $\{tP(\theta), s \leq t \leq 1\}$.
Consider the path $\gamma_{\theta} := \{(re^{i\theta},\frac{1}{re^{i \theta}}), s \leq r \leq \frac{1}{s} \}$ which is on the initial surface $\{\xi \eta = 1 \}$,
and passes through none of the characteristic points.
Thus, by the Cauchy-Kovalevskaya Theorem, $v$ is analytic in a neighborhood of each point on $\gamma_\theta$.
Choose $\varepsilon_1 >0$ small enough so that $v$ is analytic in a $\varepsilon_1$-neighborhood of $\gamma_{\theta}$.
Let $\Omega_0$ denote this tubular ($\mathbb{C}^2$) domain of analyticity.

For $1 \geq T \geq s$, let $N_{\varepsilon_2}(T)$ denote the $\varepsilon_2$-neighborhood of the segment $\{tP(\theta), T \leq t \leq 1\}$.
Since for each $1 \geq t \geq s$, the characteristic lines through $t P(\theta)$ also intersect $\gamma_{\theta}$, then for a small enough $\varepsilon_2$, 
any characteristic line that intersects $N_{\varepsilon_2}(T)$ also intersects $\Omega_0$.
Let $\Omega_T$ be the set $\left(\text{co}(\Omega_0 \cup N_{\varepsilon_2}(T)) \setminus \overline{ \text{co}(\Omega_0) }\right) \cup \Omega_0$,
where $\text{co}(S)$ denotes the convex hull of a set $S$.

\begin{claim}
 For points on $\partial \Omega_T \setminus \partial \Omega_0$, the tangent plane is a supporting hyperplane for $\Omega_T$.
\end{claim}
\begin{proof}[proof of Claim]
By definition, $\Omega_T \subset \text{co}(\Omega_0 \cup N_{\varepsilon_2}(T))$, and these two sets share a boundary near points $p \in \partial \Omega_T \setminus \partial \Omega_0$.
The tangent plane at $p \in \partial \Omega_T \setminus \partial \Omega_0$ is also a tangent plane for $\partial \text{co}(\Omega_0 \cup N_{\varepsilon_2}(T))$.
By convexity, it must be a supporting hyperplane for $\text{co}(\Omega_0 \cup N_{\varepsilon_2}(T))$.  
It is then also a supporting hyperplane for the subset $\Omega_T$.
\end{proof}

Let $E:=\{T: v$ can be analytically continued to $\Omega_T \}$.
Since $1 \in E$, $E$ is non-empty.
We will show that $E$ is both open and closed relative to $[s,1]$ and is therefore equal to $[s,1]$.
The fact that $E$ is closed follows from the fact that the domains $\Omega_T$ are continuous and nested.
To see that $E$ is open, we apply Zerner's Theorem.
Suppose $T \in E$, i.e., $v$ extends to $\Omega_T$.
By the Claim, the tangent plane $P$ to $\Omega_T$ at $p \in \partial \Omega_T \setminus \partial \Omega_0$ is a supporting hyperplane.
We must have that $P$ passes through $N_{\varepsilon_2}(s)$.
Otherwise, $P$ is a supporting hyperplane for both $\Omega_0$ and $N_{\varepsilon_2}(s)$
and, therefore, for any segment joining points in each of these sets (a contradiction).
Since $P$ passes through $N_{\varepsilon_2}(s)$ and not $\Omega_0$, it is non-characteristic.
By Theorem \ref{Zerner}, $v$ extends to a neighborhood of $p$.

CASE 2: Suppose $\theta = 0$ or $\theta = \pi$. 
For specificity, say $\theta = 0$.
Then $\gamma_0:= \{(r,\frac{1}{r}), s \leq r \leq \frac{1}{s} \}$ passes through the characteristic point $(1,1)$.
We have already observed, though, that $v$ is analytic in a neighborhood of the point $(1,1)$.
If $s \leq 1/R$, then $\gamma_0$ also passes through the characteristic points $(R,1/R)$, and $(1/R,R)$.
So, we let $s > 1/R$.  Then we can still choose an $\varepsilon_1 >0$ small enough that $v$ is analytic in a $\varepsilon_1$-neighborhood of $\gamma_0$.
We use $\Omega_0$ again to denote this domain of analyticity.
We can proceed in the same way as in the previous case, defining $N_{\varepsilon_2}(T)$ and $\Omega_T$,
except now the axis of symmetry $z=w$ intersects the advancing boundary of $\Omega_T$ for every value of $T$.
Zerner's Theorem fails at this point of intersection, 
but Theorem \ref{EKS} applies since the complex line $z=w$ is transversal to each of the boundaries $\partial \Omega_T$.
Thus, we can again prove that the set $E$ is open and closed relative to $[s,1]$,
but recall that we assumed $s > 1/R$.
\end{proof}

The method of proof can clearly be applied to other examples having axial symmetry.
In a future study, we hope to apply $\mathbb{C}^n$ techniques to some surfaces of degree four that do not have axial-symmetry.
For now, we state as a conjecture what we expect for one such example (for simplicity we formulate it in $\mathbb{R}^3$).

\begin{conj}
 Let $W({\bf x})$ be the Schwarz potential of the boundary $\Gamma$ of the domain 
$\Omega:=\{(x,y,z) \in \mathbb{R}^3 : (x^2 + y^2 + z^2)^2 - (a^2x^2 + b^2y^2 + c^2z^2) < 0\}$, with $a>b>c>0$.
Then $W$ can be analytically continued throughout $\Omega \setminus B$ where $B:=\{z=0,(x^2+y^2)^2 - \frac{x^2}{4(a^2-c^2)} - \frac{y^2}{4(b^2-c^2)} < 0 \}$. 
\end{conj}

In other words, the conjecture says that the singularity set of the Schwarz potential of $\Omega$, a three-dimensional version of the Neumann oval, 
can be confined to a set in the $xy$-plane bounded by a (two-dimensional) Neumann oval.

\section{Quadrature domains}
\label{sec:QD}

In order to limit the number of definitions in the exposition of our main results, we have so far avoided explicit mention of ``quadrature domains'', 
but it would be remiss not to discuss this important connection.  
Also, this will allow us to give a detailed approximate description of the second remark made after the statement of Theorem \ref{Neumann}.

First we consider the plane.  
A domain $\Omega$ is a \emph{quadrature domain} if it admits a formula expressing the area integral of 
any analytic function $f$ belonging to, say $L^1(\Omega)$, as a finite sum of weighted point evaluations of the function and its derivatives.
i.e. $$\int_{\Omega} {f dA} = \sum_{m=1}^{N} \sum_{k=0}^{n_k}{a_{mk}f^{(k)}(z_m)}$$
where $z_i$ are distinct points in $\Omega$ and $a_{mk}$ are constants independent of $f$. 

Suppose $\Omega$ is a bounded, simply-connected domain with non-singular, analytic boundary.
Then the following are equivalent.  
Moreover, there are simple formulas relating the details of each.

(i) $\Omega$ is a quadrature domain. \\
(ii)  The exterior logarithmic potential of $\Omega$ is equivalent to that which is generated by finitely many interior points (allowing multipoles).\\
(iii) The Schwarz function of $\partial \Omega$ is meromorphic in $\Omega$. \\
(iv)  The conformal map from the disk to $\Omega$ is rational.\\

For the equivalence of (i) and (iii), see \cite[Ch. 14]{Davis}.  For the equivalence of (i), (ii), and (iv), see \cite[Ch. 3]{Etingof}.

In higher dimensions, one simply replaces ``analytic'' with ``harmonic'' in the definition of quadrature domain.
In condition (ii), ``logarithmic'' becomes ``Newtonian''. 
In higher dimensions, ``multipole'' refers to a finite-order partial derivative of the fundamental solution to Laplace's equation.
In condition (iii), ``Schwarz function'' becomes ``Schwarz potential'', 
and instead of ``meromorphic'' the Schwarz potential must be real-analytic except for finitely many ``multipoles'' (as described above).
Then the equivalence of (i), (ii), and (iii) persists in higher dimensions (see \cite[Ch. 4]{KhSh}).
Condition (iv) of course does not extend.

If the initial domain of a Laplacian growth is a quadrature domain, then it will stay a quadrature domain 
by virtue of the equivalence of (i) and (iii) combined with Theorem \ref{LGSP}.
Moreover, according to the formulas (omitted here) relating the details of (i) and (iii), 
the consequent time-dependence of the quadrature is the content of Richardson's Theorem.
In the plane, the quadrature domain can be reconstructed from its quadrature formula, and quadrature domains are dense within natural classes Jordon curves; 
the smoother the class, the stronger the topology in which they are dense (see \cite{Bell} and the references therein).

\begin{thm}[Richardson]
\label{Rich}
If $\Omega_t$ is a Laplacian growth with $m$ sinks located at $x_i$ with rates $Q_i$, then for any harmonic function $u$
$$\frac{d}{dt} \int_{\Omega_t}{u dV} = -\sum_{i=1}^{m}{Q_i u(x_i)}$$
\end{thm}

If the initial domain is not a quadrature domain, 
then the connection of Theorem \ref{LGSP} to Richardson's Theorem requires defining quadrature domains \emph{in the wide sense},
allowing the quadrature formula to consist of a distribution with compact support contained in $\Omega$ (see \cite{KhSh} and \cite{Shapiro}).
For such generalized quadrature domains, a distribution with minimal support is called a ``mother body'' for the domain.
The singularity set of the Schwarz potential gives a supporting set for the ``mother body''.

Work of Gustafsson and Sakai guarantees existence of a quadrature domain in $\mathbb{R}^n$ satisfying a prescribed quadrature formula,
but besides the special examples in $\mathbb{R}^4$ the only explicit example for $n > 2$ is a sphere.
Moveover, little qualitative information is known about quadrature domains in higher dimensions besides that the boundary is analytic.
For instance, it is not even known whether quadrature domains are generally algebraic (in the plane, it follows from condition (iv).
We make the following conjecture, where we mean ``quadrature domain'' in the classical, restricted sense 
(otherwise the statement is trivial, since any analytic, non-singular surface is a quadrature domain in the wide sense):
\begin{conj}\label{conj:quad}
    In dimensions greater than two, there exist quadrature domains that are not algebraic.
\end{conj}

For the three-dimensional example from Theorem \ref{Neumann},
we were able to isolate the singularities for the Schwarz potential to a segment inside.
Thus, $\Omega$ is a quadrature domain in the wide sense and has a mother body supported on this segment.
We approximate the distribution using a finite number of points on this segment.
Choosing the points $x_k = -1 + k/2$, $k=0,1,..,4$, we numerically integrate $20$ harmonic basis functions (writing them in terms of Legendre polynomials) over $\Omega$.
If we assume a quadrature formula involving point evalutations at the points $(x_k,0,0)$,
then we have an overdetermined linear system for the coefficients ($20$ equations and $5$ unknowns).
We take two surfaces, and solve the least squares problem for the coefficients (using the same $5$ points).  
Then the two surfaces can be approximately described as the boundaries of initial and final domains driven by sinks at these points, 
where the total amount removed is given by the decrease in quadrature weight.

\begin{figure}[h]
    \includegraphics[scale=.5]{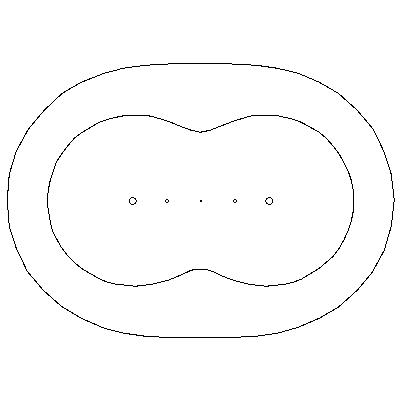}
    \caption{The profile of a supposed initial ($a=1$) and final ($a=2$) domain.  
The driving mechanism to generate the smaller domain starting from the larger can be approximated by certain amounts of suction at the indicated points.}
    \label{Approx}
\end{figure}

Suppose $\Omega_{\text{initial}}$ is given by $a=2$ (see statement of Theorem \ref{Neumann}) and $\Omega_{\text{final}}$ is given by $a=1$.
Then of the total volume extracted, according to the approximate description
$81\%$ is removed at the points $(\pm 1,0,0)$, $15\%$ at the points $(\pm 1/2,0,0)$, 
and $4\%$ at the origin (See Figure \ref{Approx}).
The accuracy of this description is reflected in the fact that the norm of the error vector for both least squares problems is on the order of $10^{-4}$.

\section{Concluding remarks}

{\bf 1.} The equivalent definitions of quadrature domains listed in Section \ref{sec:QD} indicate the possible reformulations of the Laplacian growth problem either
in terms of potential theory or in terms of holomorphic PDEs.
The potential theory approach has attracted more attention and has certain advantages such as weak formulations of Laplacian growth.
We have focused on the holomorphic PDE approach, 
and in Section \ref{sec:Cn} we gave a glimpse of its main advantage: $\mathbb{C}^n$ techniques.

{\bf 2.} The remarks at the end of Section \ref{sec:EG} mention a consequence of the Schwarz potential conjecture regarding Laplacian growth.
It would be interesting if one could obtain a partial result in the other direction along the lines of ``Surfaces satisfying the SP conjecture are preserved by Laplacian growth''.
This would only be interesting in higher dimensions, since the conjecture is already known to be true in the plane.

{\bf 3.} In Section \ref{sec:EG}, the discussion centered around the case when $\alpha = \lambda \rho = 1$ is constant.
It is natural to consider when $\alpha$ is a (fixed) non-constant entire function,
and ask if the solution $q$ to div$(\alpha \nabla q) = 1$ generalizes the data $\frac{1}{2}||{\bf x}||^2$ in the Schwarz potential conjecture.
We make the following ``elliptic Schwarz potential conjecture''.

\begin{conj}\label{conj:genSP}
 Suppose $\alpha$ is entire and that $u$ solves the Cauchy problem on a nonsingular analytic surface for div$(\alpha \nabla u) = 0$ with entire data.
Then the singularity set of $u$ is contained in the singularity set of $v$,
where $v$ solves the Cauchy problem with data $q$ the solution of div$(\alpha \nabla q) = 1$.
\end{conj}

One might object to generalizing unresolved conjectures. 
We should point out that the Schwarz potential conjecture is true in the plane and simple to prove, 
whereas we do not know if Conjecture \ref{conj:genSP} is true in the plane.
One piece of evidence for the SP conjecture is that the Schwarz potential 
developes singularities at every characteristic point of the initial surface \cite[Proposition 11.3]{DK}.
A similar proof shows that this is also true for $v$, 
where $\{\phi=0\}$ being characteristic for the elliptic operator means $\nabla \alpha \cdot \nabla \phi + \alpha \nabla \phi \cdot \nabla \phi = 0 $.

{\bf 4.} At the end of Section \ref{sec:r4} we have mentioned the fact that ``injection is independent of the order of work of sources and sinks''.
In other words, the Laplacian growths driven by different sources and sinks ``commute'' with eachother.
We can even consider, say hypothetically, injection at each of infinitely many interior points of a domain.
Then we have infinitely many processes that commute with eachother.
This, and especially its infinitesimal version which follows from the Hadamard variational formula, has the form of an ``integrable hierarchy''.
To use the preferred language in this setting, we have a ``commuting set of flows with respect to infinitely many generalized times'' 
(the ``times'' are the amounts that have been injected into each of infinitely many sources).
This holds in arbitrary dimensions 
but has recently attracted attention in two dimensions 
where it is directly connected to certain integrable hierarchies in soliton theory (see \cite{MWZ} and \cite{Teo}).
Aspects of the higher-dimensional case and possible connections to other integrable systems seem completely unexplored.

{\bf 5.} Quadrature domains have also appeared, often only implicitly, in solutions of Euler's equations.
Physically, this area of fluid dynamics is much different, involving inviscid flow with vorticity.
D. Crowdy has given a survey \cite{Crowdy} of his own work and others' (mainly in the two-dimensional case) where quadrature domains have been applied to vortex dynamics.

The ellipsoid is an example of a quadrature domain in the wide sense for which the mother body has been calculated (see \cite[Ch. 5]{KhSh}).
The exterior gravitational potential of a uniform ellipsoid coincides with that of a non-uniform density supported on the two-dimensional ``focal ellipse'' of the ellipsoid.
This fact was used by Dritschel et al \cite{Drit} as a main step in developing a model for interaction of ``quasi-geostrophic'' meteorological vortices.
Actually, they didn't use the exact density of the mother body, 
but only the location of its support in order to choose a small number of point vortices that
generate a velocity field approximating that of an ellipsoid of uniform vorticity.
Determining the strength of the approximating point vortices is nothing more 
than interpolating the quadrature formula.
Our calculation at the end of Section \ref{sec:QD}, and similar calculations, 
could have promise for extending the model in \cite{Drit} to examples of non-ellipsoidal vortices.
An important missing ingredient here is a stability analysis, which has been carried out for ellipsoids.

{\bf 6.} Our intuition for Conjecture \ref{conj:quad} is based on two suspicions regarding the axially-symmetric case.
(1) According to the singularities of the four dimensional rotation of a limacon considered in Section \ref{sec:r4},
the quadrature formula involves point evaluations up to a second-order partial derivative.
On the basis of L. Karp's procedure described in Section \ref{sec:r4}, it seems that an axially-symmetric example involving only a point evaluation of the function and a first-order partial with respect to $x$ 
will have to be generated by a curve whose Schwarz function has an essential singularity at the origin.
Then, the conformal map would be transcendental.
(2) In $\mathbb{R}^3$ we expect the situation to be at least as bad.  
Following \cite[Ch.s 4 and 5]{Gar}, one can write an integral formula involving a Gauss hypergeometric function for the solution of a Cauchy problem for an $n$-dimensional
axially-symmetric potential.
The three dimensional case of the formula has the same form as the four-dimensional case, except the involved hypergeometric function is transcendental instead of rational.

{\bf Acknowledgement:}  I wish to thank the Los Alamos National Laboratory and Mark Mineev, 
who provided an inspiring introduction to this area during my visits.
I would also like to thank Dmitry Khavinson, Razvan Teodorescu, and Darren Crowdy for helpful discussions and suggestions.

\bibliographystyle{amsplain}

\end{document}